\documentclass[a4paper]{amsart}

\usepackage[british,english]{babel} 
\usepackage{graphics,color,pgf}
\usepackage{epsfig}
\usepackage[ansinew]{inputenc}
\usepackage[all]{xy}
\usepackage{hyperref}
\usepackage{amssymb, amsmath,amsthm, mathtools, amscd}
\usepackage{mathrsfs}
\usepackage{stmaryrd}
\usepackage{cancel} 
\usepackage{tikz}
\usepackage{tikz-cd}
\usetikzlibrary{matrix}

\newtheorem{intro_thm}{Theorem}
\newtheorem{intro_cor}[intro_thm]{Corollary}
\newtheorem{intro_prop}[intro_thm]{Proposition}

\newtheorem{lem}{Lemma}[section]
\newtheorem{thm}[lem]{Theorem} 
\newtheorem{prop}[lem]{Proposition}
\newtheorem{cor}[lem]{Corollary}

\theoremstyle{definition}
\newtheorem{dfn}[lem]{Definition}

\theoremstyle{remark}

\newtheoremstyle{TheoremNum}
        {0.2 cm}{0.2 cm}              
        {\itshape}                      
        {}                              
        {}                     
        {.}                             
        { }                             
        {\thmname{\bfseries #1}\thmnote{ \bfseries #3}}
    \theoremstyle{TheoremNum}

\DeclareMathOperator{\Ker}{Ker}

\DeclareMathOperator{\C}{\mathrm{C}} 

\DeclareMathOperator{\alt}{\mathrm{alt}}
\DeclareMathOperator{\nalt}{\mathrm{n-alt}}

\renewcommand{\leq}{\leqslant}
\renewcommand{\geq}{\geqslant}
\renewcommand{\setminus}{\smallsetminus}
\title[Continuous cochains on Furstenberg boundaries]{Continuous cochains on Furstenberg boundaries and injectivity of the comparison map}

\author[Michelle Bucher]{
Michelle Bucher} 
\address{Universit\'e de Gen\`eve,  Geneva, Switzerland }
\email{Michelle.Bucher@unige.ch}

\author[Alessio Savini]{
Alessio Savini}
\address{University of Milano Bicocca, Milan, Italy} 
\email{alessio.savini@unimib.it}


\begin{document}

\begin{abstract}
 In \cite{Monod} Monod proved that any continuous cohomology class of a semisimple Lie group $G$ can be represented by a measurable  cocycle on the associated Furstenberg boundary, which we upgraded to an alternating  cocycle in \cite{BucSavAlt}. In the current paper we improve that result by showing that we can actually take a representing cocycle which is continuous on an explicit subset of generic tuples. We give an analogous result in the case of bounded cohomology. Finally, we exploit this characterization to prove the injectivity of the comparison map in degree $3$ for $\mathrm{Isom}^\circ(\mathbb{H}_{\mathbb{C}}^n)$, when $n \geq 2$, and in degree $4$ for $\mathrm{Isom}^\circ(\mathbb{H}^n_{\mathbb{R}})$, when $n \geq 2$. 
\end{abstract}

\maketitle
\section{Introduction}

Let $G$ be a connected semisimple Lie group. It is a deep result by Austin and Moore  \cite[Theorem A]{AM} that the real valued continuous cohomology of $G$ is isomorphic to its measurable variant, namely $H^*_c(G) \cong H^*_m(G)$. The analogous result for bounded cohomology also holds, that is $H^*_{c,b}(G) \cong H^*_{m,b}(G)$ \cite[Proposition 7.5.1]{Monodbook}. In the bounded setting, one can prove that the latter bounded group cohomology is realizable on the Furstenberg boundary in the following sense: Let $P<G$ be a minimal parabolic subgroup and define
$$H^q_{m,b}(G \curvearrowright G/P):=H^q(L^\infty((G/P)^{*+1})^G),$$
where the cocomplex is endowed with the homogeneous coboundary operator. Evaluation on a base point induces a map
$$ev:H^q_{m,b}(G \curvearrowright G/P)\longrightarrow H^q_{m,b}(G)$$
which can easily be shown not to depend on the base point and further, exploiting the amenability of $P$, to be an isomorphism \cite[Corollary 7.5.9]{Monodbook}.  The analogous statement in the unbounded setting was recently shown to fail in higher rank. Defining 
$$H^q_{m}(G \curvearrowright G/P):=H^q(L^0((G/P)^{*+1})^G),$$
Monod showed in \cite{Monod} that the evaluation  map
$$ev:H^q_{m}(G \curvearrowright G/P)\longrightarrow H^q_{m}(G)$$
is surjective with a kernel that can explicitly be determined in terms of some invariant subspaces of $H^\ast_c(A)$, where $A<P$ is a maximal split torus. As this kernel becomes smaller when considering the corresponding cocomplexes with the additional condition that the cochains are alternating \cite{BucSavAlt}, from now on we will only consider alternating cochains. This has no effect on the measurable or continuous (bounded) cohomology of $G$, but $H^q_{m,\mathrm{alt}}(G \curvearrowright G/P)$ is in general a strict direct summand of $H^q_{m}(G \curvearrowright G/P)$.

In this paper we investigate the validity of the analogue of Austin and Moore's result for the measurable cohomology of $G$ on $G/P$. It is certainly too optimistic to expect an isomorphism with the continuous cohomology obtained by considering $G$-invariant continuous cochains on $G/P$. By continuity, such functions would automatically be bounded, but we exhibited in \cite{BucSavExp} several examples of cohomology classes in $H^*_{m}(G \curvearrowright G/P)$ which cannot be represented by bounded cocycles. Moreover, one of the simplest example of cocycles on $G/P$, namely the hyperbolic volume $\mathrm{Vol}_n$, assigning to every $(n+1)$-tuple $(x_0,\ldots,x_n)$ of points in $\partial \mathbb{H}^n_{\mathbb{R}}$ the signed volume of the associated ideal simplex, cannot be neither continuous nor cohomologous to a continuous cocycle. Indeed, given the existence of loxodromic isometries, any continuous function on $(\partial\mathbb{H}^n_\mathbb{R})^{n+1}$ has to be constant, and a constant alternating cochain is trivial. Nevertheless, the hyperbolic volume is continuous on tuples of pairwise distinct points. We will see that this phenomenon actually holds in full generality: with the right notion of \emph{genericity}, we will prove that every measurable cocycle on $G/P$ is actually cohomologous to another representative which is continuous on a full measure dense subset of tuples. 


We postpone to Section \ref{sec barycenter} the precise Definition \ref{Def generic} and we denote by $(G/P)^{(q+1)}$ the subset of generic tuples. For now we record the following facts:
\begin{itemize}
\item  The subset $(G/P)^{(q+1)}$ is open, dense and of full measure in $(G/P)^{q+1}$. 
\item  When $G$ has real rank equal to one, the subset of generic tuples is precisely the subset of pairwise distinct tuples. 
\end{itemize}

We define the \emph{continuous alternating cohomology} of the $G$-action on $G/P$ as the cohomology of the cocomplex
$$
H^q_{c,\alt}(G \curvearrowright G/P):=H^q(C_{c,\alt}((G/P)^{(*+1)})^G),
$$
where 
$$
C_{c,\alt}((G/P)^{(q+1)})^G:=\{ f : (G/P)^{(q+1)} \rightarrow \mathbb{R} \ | \ \textup{$f$ is continuous and alternating} \}^G
$$
is endowed with the homogeneous coboundary operator.

\begin{intro_thm}\label{thm_continuous}
Let $G$ be a connected semisimple Lie group with finite center and $P<G$ be a minimal parabolic subgroup. The inclusion of cocomplexes 
$$
C_{c,\alt}((G/P)^{(q+1)})^G \longrightarrow L^0_{\alt}((G/P)^{q+1})^G 
$$
induces a surjection 
$$
\xymatrix{
H^q_{c,\alt}(G \curvearrowright G/P) \ar@{->>}[r]&  H^q_{m,\alt}(G \curvearrowright G/P).
}
$$
\end{intro_thm}

As a consequence any class in the continuous cohomology of $G$ can  actually be realized as a $G$-invariant alternating cocycle on the boundary $G/P$ which is continuous on the subset of generic tuples. When $G$ has real rank equal to one this subset is the subset of pairwise distinct tuples. 

We do not know if the surjection of Theorem \ref{thm_continuous} is an isomorphism in general, as one could wish for when looking for an anologue of the result by Austin and Moore. The situation improves in the bounded setting. Defining the \emph{continuous bounded alternating cohomology} of the $G$-action on $G/P$ as 
$$
H^q_{c,b,\alt}(G \curvearrowright G/P):=H^q(C_{c,b,\alt}((G/P)^{(\ast+1)})^G),
$$
we obtain the desired isomorphism: 

\begin{intro_thm}\label{thm_continuous_bounded}
Let $G$ be a connected semisimple Lie group with finite center and $P<G$ be a minimal parabolic subgroup. The inclusion of cocomplexes 
$$
C_{c,b,\alt}((G/P)^{(q+1)})^G \longrightarrow L^\infty_{\alt}((G/P)^{q+1})^G 
$$
induces an isomorphism
$$H^q_{c,b,\alt}(G \curvearrowright G/P)\cong  H^q_{m,b,\alt}(G \curvearrowright G/P) .$$
\end{intro_thm}

To prove Theorems \ref{thm_continuous} and \ref{thm_continuous_bounded} we adapt the spectral sequences from \cite{Monod} and \cite{BucSavAlt} to the continuous setting. Most of the adaptations are straightforward except for two crucial issues: 1) In the unbounded case, it is not clear that one of the spectral sequence associated to the double complex we consider vanishes, which is the reason why we cannot exploit that the other spectral sequence degenerates and conclude that there is an isomorphism with the  measurable cohomology. 2) The proof of the triviality of the columns $q\geq 3$ in the first page of the first spectral sequence by Monod \cite[Proposition 5.1]{Monod} does not work in the continuous setting. Our proofs (Theorem \ref{thm vanishing gp unbounded} and Theorem \ref{thm vanishing gp bounded}) will depend on the existence of a $G$-equivariant symmetric continuous \emph{barycenter map} \cite{BucSavProj} 
\begin{equation}\label{eq intro bar 3}
\mathrm{bar}_3:(G/P)^{(3)} \rightarrow G/K.
\end{equation}
Here $G/K$ is the Riemannian globally symmetric space associated to $G$. When $G=\mathrm{Isom}^\circ(\mathbb{H}^n_\mathbb{R})$, the barycenter is precisely the barycenter of the ideal triangle defined by $x,y,z \in \partial \mathbb{H}^n_{\mathbb{R}}$. 

\subsection*{Injectivity of the comparison map} The forgetful functor induces so called \emph{comparison maps} which naturally fit in the commutative diagram

\begin{equation}\label{diagram comparison intro}
\xymatrix{
 H^q_{c,b,\alt}(G \curvearrowright G/P) \ar[d]_{\mathrm{comp}^q_{G \curvearrowright G/P}} \ar@{->}[rr]^{ ev}_{\cong} &&H^q_{c,b}(G) \ar[d]^{\mathrm{comp}^q_G}\\
H^q_{c,\alt}(G \curvearrowright G/P)\ar@{->}[rr]^{ ev} && H^q_c(G).
}
\end{equation}

It is a mysterious open question whether the function $\mathrm{comp}_G^q$ is actually an isomorphism \cite[Problem A]{MonodICM}. In the current manuscript we focus our attention on the injectivity issue. The injectivity of $\mathrm{comp}^q_G$ is trivial in degrees $0$ and $1$. Injectivity is known to hold:
\begin{itemize}
\item In degree $2$ in full generality \cite{BM99}.
\item In degree $3$ for $\mathrm{SL}(2,\mathbb{R})$ \cite{BuMo2}, for $\mathrm{SL}(2,\mathbb{C})$ \cite{Bloch}, for $\mathrm{SL}(n,\mathbb{R})$ \cite{MonodSLn}, for $\mathrm{SL}(n,\mathbb{C})$ \cite{MonodSLn,BBI18}, for $\mathrm{Isom}^\circ(\mathbb{H}^n_{\mathbb{R}})$ \cite{Pieters}, for $\mathrm{SO}(n,\mathbb{C})$ and $\mathrm{Sp}(2n,\mathbb{C})$ \cite{DLC}, for products of groups of orientation preserving isometries of hyperbolic spaces \cite{BucSavExp}. 
\item In degree $4$ only for $\mathrm{SL}(2,\mathbb{R})$ \cite{HO}.
\end{itemize} 

In this paper we further establish injectivity in degree $3$ for $\mathrm{Isom}^\circ(\mathbb{H}^n_\mathbb{C})$ (Corollary \ref{cor complex}) and in degree $4$ for  $\mathrm{Isom}^\circ(\mathbb{H}^n_\mathbb{R})$ (Theorem \ref{thm real}).

From Diagram (\ref{diagram comparison intro}) it is clear that $\mathrm{comp}^q_G$ is injective if and only if $\mathrm{comp}^q_{G \curvearrowright G/P}$ is injective and the kernel of $ev$ intersects the image of $\mathrm{comp}^q_{G \curvearrowright G/P}$ trivially. We will see that exploiting the proofs of Theorems \ref{thm_continuous} and \ref{thm_continuous_bounded} we can do a little better by restricting the latter condition on the  isomorphic copy of 

$$H^{q-1}_c(A)^{w_0}=\mathrm{Ker}(ev: H^q_{m,\mathrm{alt}}(G \curvearrowright G/P)\longrightarrow H^q_{m}(G)),$$
which naturally injects in the kernel of the evaluation map  
$$ev: H^q_{c,\mathrm{alt}}(G \curvearrowright G/P)\longrightarrow H^q_{c}(G).$$
More precisely, we establish the following criterion: 

\begin{intro_prop}\label{prop injectivity comparison}\label{criterion}
Let $G$ be a connected semisimple Lie group with finite center. Let $P$ be a miminal parabolic subgroup, $A<P$ a maximal split torus and $w_0$ a representative of the longest element in the Weyl group. The comparison map 
$$
\mathrm{comp}^q_G:H^q_{c,b}(G) \rightarrow H^q_c(G)
$$
is injective if and only if 
\begin{enumerate}
\item the kernel of the measurable evaluation map consists of unbounded classes in $H^q_{c,\alt}(G \curvearrowright G/P)$ , i.e. 
$$H^{q-1}_c(A)^{w_0}\cap \mathrm{comp}^q_{G \curvearrowright G/P}(H^q_{c,b,\alt}(G \curvearrowright G/P))=\{0\},$$ 

\item the comparison map 
$$
\mathrm{comp}^q_{G \curvearrowright G/P}: H^q_{c,b,\alt}(G \curvearrowright G/P) \longrightarrow H^q_{c,\alt}(G \curvearrowright G/P)
$$
is injective.
\end{enumerate}

\end{intro_prop}

The first condition of this criterion is trivially satisfied for $q$ strictly greater than the real rank of $G$. It also always holds for even $q$ when $w_0$ acts as $-1$ on the Lie algebra of $A$, since in that case $H^{q-1}_c(A)^{w_0}=0$. It is our belief that the validity of the first condition is the easier part of this criterion. The difficulty of proving injectivity of the comparison map has thus shifted to the boundary. Exploiting continuity and transitivity properties of the action of $\mathrm{Isom}^\circ(\mathbb{H}^n_{\mathbb{R}})$ on its Furstenberg boundary we prove: 

\begin{intro_thm}\label{thm boundary deg 4}
Let $G=\mathrm{Isom}^\circ(\mathbb{H}^n_{\mathbb{R}})$, where $n \geq 2$. If $P<G$ is any parabolic subgroup, the map
$$
\xymatrix{
\mathrm{comp}^4_{G \curvearrowright G/P}:H^4_{c,b}(G \curvearrowright G/P) \ar@{^{(}->}[r] & H^4_c(G \curvearrowright G/P)
}
$$
is injective. 
\end{intro_thm}

As a consequence of Proposition \ref{prop injectivity comparison} and Theorem \ref{thm boundary deg 4} we obtain 

\begin{intro_thm}\label{thm group deg 4}\label{thm real}
Let $G=\mathrm{Isom}^\circ(\mathbb{H}^n_{\mathbb{R}})$, where $n \geq 2$. The comparison map 
$$
\xymatrix{
\mathrm{comp}^4_{G}:H^4_{c,b}(G) \ar[r]^{\hspace{10pt} \cong} & H^4_c(G)
}
$$
is an isomorphism. 
\end{intro_thm}

This was known previously only in the case of $n=2$ by Hartnick and Ott  \cite{HO}. Note that the surjectivity is well known.

\subsection*{Configurations of triples of points and injectivity in degree $3$} We conclude with a curious interplay between an algebraic property of the action of the longest element $w_0$ and the topology (more precisely only the non-compactness) of the configuration space of triples of generic points:

\begin{intro_prop}\label{intro prop injective degree 3}
Let $G$ be a connected semisimple Lie group with finite center. Let $P<G$ be a minimal parabolic subgroup, $A<P$ a maximal split torus and $w_0$ a representative of the longest element in the Weyl group. \begin{enumerate}
\item If the quotient $G\backslash (G/P)^{(3)}$ is compact, then the comparison map
$$
\mathrm{comp}^2_{G \curvearrowright G/P}:H^2_{c,b}(G \curvearrowright G/P) \longrightarrow H^2_c(G \curvearrowright G/P)
$$
is surjective and 
$$
\mathrm{comp}^3_{G \curvearrowright G/P}:H^3_{c,b}(G \curvearrowright G/P) \longrightarrow H^3_c(G \curvearrowright G/P)
$$
is injective. 
\item If $w_0$ does not act as $-1$ on $\mathfrak{a}=\mathrm{Lie}(A)$, then $G\backslash (G/P)^{(3)}$ is not compact.
\end{enumerate}
\end{intro_prop}

\begin{proof} \begin{enumerate} \item A $G$-invariant continuous cochain $f:(G/P)^{(3)}\rightarrow \mathbb{R}$ corresponds to a continuous map $F:G\backslash (G/P)^{(3)}\rightarrow \mathbb{R}$. If $G\backslash (G/P)^{(3)}$ is compact, such a map is automatically bounded, and thus in this case any $G$-invariant continuous cochain $f:(G/P)^{(3)}\rightarrow \mathbb{R}$ is bounded. This immediately proves surjectivity in degree $2$. For the injectivity in degree $3$, suppose that a bounded $G$-invariant continuous cocycle $b: (G/P)^{(4)}\rightarrow \mathbb{R}$ vanishes in $H^3_c(G \curvearrowright G/P)
$. This means that there exists a $G$-invariant continuous cochain $f:(G/P)^{(3)}\rightarrow \mathbb{R}$ such that $\delta f=b$. But if $G\backslash (G/P)^{(3)}$ is compact, the cochain $f$ is bounded. 
\item We will use injectivity of the comparison map in degree $2$  \cite{BM99}, which by Proposition \ref{criterion} in particular implies that the kernel of the evaluation map intersects the image of the comparison map 
$$
\mathrm{comp}^2_{G \curvearrowright G/P}:H^2_{c,b}(G \curvearrowright G/P) \longrightarrow H^2_c(G \curvearrowright G/P)
$$
trivially. Now $w_0$ does not act as $-1$ if and only if $H^1_c(A)^{w_0}\neq 0$. In particular the kernel of the evaluation map is nontrivial and as a consequence the latter comparison map is not surjective. By the first item of the proposition this implies that the quotient $G\backslash (G/P)^{(3)}$ is not compact. 
\end{enumerate}
\end{proof}

For example, for $G=\mathrm{Isom}^\circ(\mathbb{H}^2_\mathbb{R})$, there are two orbits of generic triple of points in the Furstenberg boundary $\partial \mathbb{H}^2_\mathbb{R}$ so that the configuration space of distinct triples of points consists of $2$ points and is hence compact. For $G=\mathrm{Isom}^\circ(\mathbb{H}^n_\mathbb{R})$, with $n\geq 3$, there is only one orbit. For $G=\mathrm{Isom}^\circ(\mathbb{H}^n_\mathbb{C})$, with $n\geq 2$, the Cartan angular invariant gives a homeomorphism between the configuration space of distinct triples of points in $\partial \mathbb{H}^n_\mathbb{C}$ and the closed interval $[-\pi/2,\pi/2]$. In contrast, for $G=\mathrm{SL}(3,\mathbb{R})$, the Furstenberg boundary is given by the space of maximal flags in $\mathbb{R}^3$ and the triple ratio of generic triples of maximal flags induces a homeomorphism between the configuration space of triples and $\mathbb{R}\setminus \{0,1\}$ which  is indeed not compact. 

Note that the converse of the second item of Proposition \ref{intro prop injective degree 3} does not hold. For example for $G=\mathrm{Sp}(4,\mathbb{R})$ the longest element $w_0$ acts on $\mathfrak{a}$ as $-1$ but the quotient $G\backslash (G/P)^{(3)}$ is not compact. 

We conclude the introduction with a direct application of Proposition \ref{intro prop injective degree 3}: 

\begin{intro_cor}\label{cor complex}
Let $G=\mathrm{Isom}^\circ(\mathbb{H}^n_{\mathbb{C}})$, where $n \geq 2$. The comparison map
$$
\xymatrix{
\mathrm{comp}^3_{G}:H^3_{c,b}(G) \ar@{^{(}->}[r] & H^3_c(G).
}
$$
is injective. 
\end{intro_cor}

\begin{proof} We only need to verify the two items of Proposition \ref{criterion}: The first one holds since $A\cong \mathbb{R}$ so that $H^2_c(A)^{w_0}=0$. The second holds since, as noticed above, the space of configuration of triples of points is homeomorphic to a closed interval, so that the comparison map 
$$
\xymatrix{
\mathrm{comp}^3_{G \curvearrowright G/P}:H^3_{c,b}(G \curvearrowright G/P) \ar@{^{(}->}[r] & H^3_c(G \curvearrowright G/P).
}
$$
is injective by the first item of Proposition \ref{intro prop injective degree 3}. 
\end{proof}

\section{Generic tuples on the Furstenberg boundary and barycenter map}\label{sec barycenter}

Let $G$ be a connected semisimple Lie group with finite center. We consider a minimal parabolic subgroup $P<G$ and a maximal split torus $A<P$ with Lie algebra $\mathrm{Lie}(A)=\mathfrak{a}$. If $K$ denotes a maximal compact subgroup in $G$, recall that the \emph{Weyl group} $W$ is defined as the quotient 
$$
W:=N_K(\mathfrak{a})/Z_K(\mathfrak{a}),
$$
where $N_K$ and $Z_K$ are respectively the normalizer and the centralizer in $K$ of $\mathfrak{a}$ with respect to the adjoint representation. We fix a representative $w_0$ of the longest element in $W$. More generally, every time we pick an element $w \in W$, we will tacitly assume to fix a representative of it. 

Two points $x,y \in G/P$ in the Furstenberg boundary are said to be \emph{opposite} if they lie in the same $G$-orbit as $(P,w_0P)$. Given two opposite points $x,y \in G/P$, we denote by $F_{x,y}$ the unique maximal flat determined by $x$ and $y$. In the particular case of $(P,w_0P)$, we denote it by $F_A$ and we call it the canonical maximal flat. The \emph{boundary} $\partial F \subset G/P$ of a maximal flat is the set of equivalence classes of Weyl chambers determined by $F$. For $F_A$ we have
$$
\partial F_A = \{ wP \ | \ w \in W\}.
$$
Recall that the $G$-action on the set of maximal flats is transitive and it holds that $\partial(gF_A)=g(\partial F_A)$. 

\begin{dfn}\cite{BucSavProj} \label{Def generic} Let $q\geq 1$. We define the set $(G/P)^{(q)}$ of \emph{generic} $q$-tuples in $(G/P)^q$ as follows: 
 \begin{itemize}
\item If $q=1$, take $(G/P)^{(1)}:=G/P$. 
\item If $q=2$, it is the set of opposite tuples. 
\item For $q\geq 3$, if $w_0$ acts as $-1$ on $\mathfrak{a}$, we define the set of generic points $(G/P)^{(q)}$ to be the set of pairwise opposite points. 
\item For $q\geq 3$, if $w_0$ does not act as $-1$ on $\mathfrak{a}$, we define the set of generic points $(G/P)^{(q)}$ to consist of $q$-tuples $(x_1,\dots,x_q)\in (G/P)^q$ such that $x_k$ is opposite to every point in the boundary $\partial F_{x_i,x_j}$, for every distinct  $1 \leq i,j,k \leq q$.  
\end{itemize}
\end{dfn}

Observe that when the rank of $G$ is one, generic $q$-tuples are precisely $q$-tuples of pairwise distinct points. 

\begin{lem}\label{lem symmetrically weyl opposite}
Given a connected semisimple Lie group $G$ with finite center and a minimal parabolic subgroup $P<G$, the subset $(G/P)^{(q)}$ of generic $q$-tuples is open, dense and of full measure. 
\end{lem}

\begin{proof}
Suppose that $w_0$ does not act as $-1$ on $\mathfrak{a}=\mathrm{Lie}(A)$. We start showing that $(G/P)^{(3)}$ is an open dense subset of full measure. Given a point $x \in G/P$ we denote by 
$$
\mathrm{Opp}_x:=\{ y \in G/P \ | \ y \ \text{is opposite to} \ x \},
$$
and we define
$$
(G/P)_{\mathrm{opp}}:=\bigcap_{w \in W} \mathrm{Opp}_{wP}. 
$$
Notice that $(G/P)_{\mathrm{opp}}$ is an open and dense subset of $G/P$ of full measure, since each $\mathrm{Opp}_{wP}$ satisfies the same properties and $W$ is finite. 

The map
$$
\pi:G \times (G/P)_{\mathrm{opp}} \longrightarrow (G/P)^{(3)}, \ \ \pi(g,x):=(gP,gw_0P,gx),
$$
is an open surjection which parametrizes the subset of generic triples. Thus $(G/P)^{(3)}$ is open in $(G/P)^3$. Now we show that $(G/P)^{(3)}$ has full measure. As $g \in G$ varies, the first two coordinates of the triple $(gP,gw_0P,gx)$ cover the full measure subset of opposite pairs $(G/P)^{(2)}$ and $gx$ varies in $g(G/P)_{\mathrm{opp}}$, which has full measure in $G/P$ since the measure class is preserved by the $G$-action. A similar argument applies to show that $(G/P)^{(3)}$ is dense. 

Finally $(G/P)^{(q)}$ is an open and dense subset of full measure, being a finite intersection of subsets with the same properties.

In the case where $w_0$ acts as $-1$ on $\mathfrak{a}$, the claim follows from the fact that opposite pairs $(G/P)^{(2)}$ form an open, dense subset in $(G/P)^2$ of full measure. 
\end{proof}

The relevance of the subset $(G/P)^{(q)}$ of generic $q$-tuples relies on the existence of a \emph{barycenter map} onto the symmetric space $G/K$. More precisely, the authors \cite[Corollary 4]{BucSavProj} proved the existence of a $G$-equivariant symmetric continuous map
$$
\mathrm{bar}_q:(G/P)^{(q)} \longrightarrow G/K
$$
which we call \emph{barycenter}. We will exploit the map $\mathrm{bar}_3$ to obtain a vanishing result in the computation of the spectral sequence for continuous cochains. 

\section{Proof of Theorem \ref{thm_continuous}}\label{sec continuous theorem}

Recall that $G$ is a connected semisimple Lie group with finite center. Fix a
minimal parabolic subgroup $P<G$ and a maximal split torus $A<P$. Let $K$ be
a maximal compact subgroup and $M<K$ be the centralizer of $A$. Let $w_0$ be a
representative of the longest element in the Weyl group.

We define
\begin{equation}\label{bicomplex Kpq}
C^{p,q}:=C_c(G^{p+1} \times (G/P)^{(q)})^G,
\end{equation}
the space of $G$-invariant continuous functions on the product $G^{p+1} \times (G/P)^{(q)}$. The vertical differential 
$$
d^\uparrow:C^{p,q} \rightarrow C^{p+1,q}
$$
is the usual homogeneous differential on the $G$-variable, whereas the horizontal one 
$$
d^\rightarrow:C^{p,q} \rightarrow C^{p,q+1}
$$
is the homogeneous differential on $(G/P)^{(q)}$ weighted with the sign $(-1)^{p+1}$, so that $d^\rightarrow d^\uparrow=d^\uparrow d^\rightarrow$. 

Since $(G/P)^{(q)}$ is invariant with respect to the action of the symmetric group $\mathrm{Sym}(q)$, we have a well-defined \textit{alternation map}
\begin{equation}\label{eq alternation continuous}
\mathrm{Alt}_q:C^{p,q} \rightarrow C^{p,q},
\end{equation}
$$
\mathrm{Alt}_q(f)(g_0,\ldots,g_p,h_1P,\ldots,h_qP):=\frac{1}{q!}\sum_{\sigma \in \mathrm{Sym}(q)} \mathrm{sgn}(\sigma)f(g_0,\ldots,g_p,h_{\sigma(1)}P,\ldots,h_{\sigma(q)}P),
$$
which is an idempotent function, namely $\mathrm{Alt}^2_q=\mathrm{Alt}_q$. As a consequence $C^{p,q}$ decomposes into the direct sum of the kernel and the image of $\mathrm{Alt}_q$. We call the elements of the latter \emph{boundary alternating cochains} and we denote them by $C^{p,q}_{\mathrm{alt}}$, whereas we call the elements of the former \emph{boundary non-alternating cochains} and we denote them by $C^{p,q}_{\mathrm{n-alt}}$. 
Since $\mathrm{Alt}_q$ commutes with both the horizontal and the vertical differentials, we have a splitting of double complexes, namely
$$
C^{p,q} \cong C^{p,q}_{\mathrm{alt}} \oplus C^{p,q}_{\mathrm{n-\alt}}.
$$
A key strategy in the proof will be to study the larger bicomplex $C^{p,q}$ to get the desired information about $C^{p,q}_{\mathrm{alt}}$. By the usual exponential law on the space of continuous functions, we have an isomorphism
\begin{equation}\label{expo:continuous}
C^{p,q} \cong C_c(G^{p+1},C_c((G/P)^{(q)}))^G
\end{equation}
and similarly
\begin{equation}\label{expo:cont:alt}
C^{p,q}_{\mathrm{alt}} \cong C_c(G^{p+1},C_{c,\mathrm{alt}}((G/P)^{(q)}))^G. 
\end{equation}

By substituting continuous functions with (classes of) measurable ones we obtain another double complex
\begin{equation}\label{meas:bicomplex}
M^{p,q}:=L^0(G^{p+1} \times (G/P)^{(q)})^G,
\end{equation}
with the same horizontal and vertical differentials as $C^{p,q}$. Alternation is still well-defined and we have a direct decomposition of double complexes
$$
M^{p,q} \cong M^{p,q}_{\mathrm{alt}} \oplus M^{p,q}_{\mathrm{n-alt}}. 
$$
Thanks to Fubini-Tonelli's Theorem, $M^{p,q}$ is precisely the bicomplex introduced by Monod \cite{Monod} and $M^{p,q}_{\mathrm{alt}}$ is the subcomplex introduced by the authors \cite{BucSavAlt}. 

By Lemma \ref{lem symmetrically weyl opposite} the subset of generic $q$-tuples has full measure in $(G/P)^q$ and hence we have a well-defined inclusion of continuous functions into (classes of) measurable ones 
$$
C^{p,q} \longrightarrow M^{p,q},
$$
which respects both the vertical differentials and the horizontal ones, leading to a map of bicomplexes. Since the alternation map is the same on both continuous and measurable functions, we have a similar inclusion of bicomplexes in the alternating case
$$
C^{p,q}_{\mathrm{alt}} \longrightarrow M^{p,q}_{\mathrm{alt}}. 
$$
We will exploit this map and all the information about the bicomplex of measurables functions. 

We start by studying the first page of the spectral sequence
$$
E_1^{p,q}=(H^q(C^{\ast,p}_{\mathrm{alt}},d^\uparrow),d_1=d^\rightarrow). 
$$
By the exponential law given in Equation \eqref{expo:continuous}, the $(p,q)$-entry of this first page is given by
$$
E_1^{p,q}=H^q_c(G,C_{c,\alt}((G/P)^{(p)}). 
$$
Recall that the alternation map induces a splitting
\begin{equation}\label{eq coefficients splitting}
H^q_c(G,C_c((G/P)^{(p)})) \cong H_c^q(G,C_{c,\alt}((G/P)^{(p)})) \oplus H_c^q(G,C_{c,\nalt}((G/P)^{(p)}))
\end{equation}
on the corresponding cohomology groups. 
 
The first column ($p=0$) of the first page of our spectral sequence is clearly $E_1^{0,q}=H^q_c(G)$. Let us investigate the columns $p=1$ and $p=2$. Before doing that, we remind the reader of the existence of a natural action of $w_0$ on the cohomology $H^q_c(A)$ induced by the adjoint representation. In particular, we can consider the map 
$$
\Pi^q_{w_0}:H^q_c(A) \rightarrow H^q_c(A), \  \ \ \alpha \mapsto \alpha-\mathrm{Ad}(w_0)(\alpha). 
$$
We say that an element $\alpha \in H^q_c(A)$ is $w_0$-\emph{invariant} if it lies in the kernel of $\Pi^q_{w_0}$, whereas we say that $\alpha$ is $w_0$-\emph{equivariant} if it lies in the image of $\Pi^q_{w_0}$. We denote by 
$$
H^q_c(A)^{w_0}:=\Ker(\Pi^q_{w_0})
$$
the subspace of $w_0$-invariant classes. 

\begin{lem}\label{lem first two columns}
For every $q \geq 0$ we have that
$$
H^q_c(G,C_c(G/P)) \cong H^q_c(A),
$$
$$
H^q_c(G,C_{c,\alt}((G/P)^{(2)})) \cong \mathrm{Im}(\Pi^q_{w_0}).
$$
\end{lem} 

\begin{proof}
There exists an isomorphism of cocomplexes
\begin{equation}\label{iso shapiro P}
\Psi:C_c(G^{q+1})^P \longrightarrow C_c(G^{q+1} \times G/P)^G, \ \  (\Psi f)(g_0,\ldots,g_q,gP):=f(g^{-1}g_0,\ldots,g^{-1}g_q).
\end{equation}
If we compose $\Psi$ with the exponential law of Equation \eqref{expo:continuous}, we obtain the following isomorphism
$$
H^q_c(G,C_c(G/P)) \cong H_c^q(P),
$$
for every $q \geq 0$. Now it is sufficient to recall that the inclusion $A \rightarrow P$ induces an isomorphism in cohomology \cite[Proposition 3.1]{Monod}, namely
$$
H^q_c(P) \cong H^q_c(A),
$$
for every $q \geq 0$. 

Recall that the set $(G/P)^{(2)}$ of pairs of opposite points is an open dense subset of $(G/P)^2$ which can be identified with the quotient $G/MA$ via the map
$$
G/MA \rightarrow (G/P)^{(2)},  \ \ \ gMA \mapsto (gP,gw_0P).
$$
The existence of this $G$-equivariant diffeomorphism implies the following isomorphism
$$
H^q_c(G,C_c((G/P)^{(2)})) \cong H^q_c(G,C_c(G/MA)),
$$
for every $q \geq 0$. By substituting $P$ with $MA$ in Equation \eqref{iso shapiro P} and exploiting again the exponential law, we obtain the isomorphism
$$
H^q_c(G,C_c(G/MA)) \cong H_c^q(MA),
$$
for every $q \geq 0$. Since $M$ is compact and centralizes $A$, we can apply \cite[Theorem 9.1]{Blanc} to get 
$$
H^q_c(MA) \cong H^q_c(A),
$$
for every $q \geq 0$. Moreover, the same proof as \cite[Proposition 8]{BucSavAlt} shows that the splitting given by Equation \eqref{eq coefficients splitting} corresponds to the splitting of $H^q_c(A)$ into $w_0$-invariant and $w_0$-equivariant classes. More precisely, we have that 
$$
H^q_c(G,C_{c,\nalt}((G/P)^{(2)})) \cong H^q_c(A)^{w_0}, \ \ \ H^q_c(G,C_{\alt}((G/P)^{(2)})) \cong \mathrm{Im}(\Pi^q_{w_0}), 
$$
for every $q \geq 0$. This concludes the proof. 
\end{proof}

The next crucial step is to show triviality of the columns $E^{p,\ast}_1$ for $p\geq 3$. This is the first instance where a completely new ingredient is needed. The strategy will be to show  that the cohomology of the cocomplex $C^{\ast,p}$ is isomorphic to the cohomology of 
\begin{equation}\label{def Cpq}
C_K^{*,p}:=C_c((G/K)^{*+1} \times (G/P)^{(p)})^G.
\end{equation}
Then we will see that the latter cohomology group is trivial in the proof of Theorem \ref{thm vanishing gp unbounded}.


\begin{lem}\label{K invariant representative}
The evaluation on the base point $K$
$$C^{\ast,p}_K=C_c((G/K)^{*+1} \times (G/P)^{(p)})^G\longrightarrow  C_c(G^{*+1} \times (G/P)^{(p)})^G=C^{\ast,p}$$
induces an isomorphism in cohomology, namely
$$H^q(C^{\ast,p}_K,d^\uparrow) \longrightarrow H^q(C^{\ast,p},d^\uparrow).$$ 
\end{lem}

\begin{proof}
Since $K$ is compact, it admits a finite Haar measure $\mu_K$. We define a left inverse of the evaluation on $K$ by 
$$
\alpha_K:C_c(G^{q+1} \times (G/P)^{(p)})^G\longrightarrow C_c((G/K)^{q+1} \times (G/P)^{(p)})^G$$
by
\begin{align*}
&\alpha_K(f)(g_0K,\dots,g_qK,h_1P,\ldots ,h_pP)=\\
&\frac{1}{\mu_K(K)^{q+1}}\int_K \ldots \int_K f(g_0k_0,\ldots,g_qk_q,h_1P,\ldots,h_pP)d\mu_K(k_0)\ldots d\mu_K(k_q), \end{align*}
for $f\in C^{q,p}$. 
Notice that the integrals are all well-defined since $\mu_K$ is finite and $f$ is continuous. It is obvious that $\alpha_K$ is a left inverse to the evaluation on $K$. 

The composition of  $\alpha_K$ followed by evaluation is simply
\begin{align*}
&\alpha(f)(g_0,\ldots,g_q,h_1P,\ldots ,h_pP)=\\
&\frac{1}{\mu_K(K)^{q+1}}\int_K \ldots \int_K f(g_0k_0,\ldots,g_qk_q,h_1P,\ldots,h_pP)d\mu_K(k_0)\ldots d\mu_K(k_q), \end{align*}
for $f\in C^{q,p}$.
To prove that $\alpha$ is homotopic to the identity, we define the following chain homotopy
$$
H:C^{q,p} \rightarrow C^{q-1,p},
$$
$$
(H f)(g_0,\ldots,g_{q-1},h_1P,\ldots,h_pP):=
$$
$$
\sum_{i=0}^{q-1} (-1)^i \frac{1}{\mu_K(K)^{q-i}}\int_K \ldots \int_K f(g_0,\ldots,g_i,g_ik_i,\ldots,g_{q-1}k_{q-1},h_1P,\ldots)d\mu_K(k_i)\ldots d\mu_K(k_{q-1}). 
$$
A standard computation shows that 
$$
\alpha-\mathrm{Id}=dH+Hd,
$$
and the statement is proved. 
\end{proof}

\begin{thm}\label{thm vanishing gp unbounded}
For every $q \geq 1$ and $p \geq 3$ we have that 
$$
H_c^q(G,C_c((G/P)^{(p)})) \cong 0.
$$
\end{thm}

\begin{proof}
Thanks to Lemma \ref{K invariant representative} it is sufficient to construct a contracting homotopy for the cocomplex 
$$
(C_K^{q,p},d^\uparrow).
$$

Since in the coefficient module we have at least three points on the boundary, we will exploit the barycenter map of Section \ref{sec barycenter}. More precisely, we define
$$
H:C_K^{q,p} \rightarrow C_K^{q-1,p},
$$ 
$$
(H f)(g_0K,\dots,g_{q-1}K,h_1P,\ldots,h_pP):=
$$
$$
f(\mathrm{bar}_3(h_1P,h_2P,h_3P),g_0K,\ldots,g_{q-1}K,h_1P,\ldots,h_pP). 
$$
We claim that $H$ is well-defined. For any $f \in C_K^{q,p}$, the function $Hf$ is continuous, being the composition of continuous functions. In a similar way, $Hf$ is $G$-invariant by the $G$-equivariance of the map $\mathrm{bar}_3$. A direct computation shows that 
$$
f=H(df)+d(Hf),
$$
which precisely means that $H$ is a contracting homotopy for the cocomplex, as claimed. This concludes the proof. 
\end{proof}

\begin{cor}\label{cor vanishing 3 points}
For every $q \geq 1$ and every $p \geq 3$ we have that 
$$
H^q_c(G,C_{c,\alt}((G/P)^{(p)})) \cong 0. 
$$
\end{cor}

\begin{proof}
By Equation \eqref{eq coefficients splitting} we have a splitting 
$$
H^q_c(G,C_c((G/P)^{(p)})) \cong H_c^q(G,C_{c,\alt}((G/P)^{(p)})) \oplus H_c^q(G,C_{c,\nalt}((G/P)^{(p)})).
$$
Theorem \ref{thm vanishing gp unbounded} implies that the left-hand side of the previous equation vanishes. As a consequence, the two summands on the right-hand side must vanish as well, and the statement is proved. 
\end{proof}

\begin{proof}[Proof of Theorem \ref{thm_continuous}]

We start by computing explicitly the first page of the spectral sequence
$$
E_1^{p,q}=(H^q(C^{\ast,p}_{\mathrm{alt}},d^\uparrow),d_1=d^\rightarrow). 
$$
By Corollary \ref{cor vanishing 3 points} we have that $E_1^{p,q} \cong 0$ whenever $q \geq 1$ and $p \geq 3$. The first column $E_1^{0,q}$ is naturally isomorphic to the continuous cohomology of $G$, namely $H^q_c(G)$. By Lemma \ref{lem first two columns} we have that the columns with coefficients having either one or two points on the boundary are isomorphic to 
$$
E^{1,q}_1 \cong H^q_c(A), \ \ \ E^{2,q}_1 \cong \mathrm{Im}(\Pi^q_{w_0}),
$$
for $q \geq 1$.
Finally, the bottom row coincides with the $G$-invariant continuous alternating functions on $(G/P)^{(p)}$, that is
$$
E^{p,0}_1 \cong C_{c,\alt}((G/P)^{(p)})^G.
$$ 

\begin{figure}[!h]
\centering
\begin{tikzpicture}
  \matrix (m) [matrix of math nodes,
             nodes in empty cells,
             nodes={minimum width=9ex,
                    minimum height=9ex,
                    outer sep=-3pt},
             column sep=2ex, row sep=-1ex,
             text centered,anchor=center]{
         q      &         &          &          & \\
          \cdots & \cdots & \cdots & \cdots   & \cdots \\
          3    & \ H^3_c(G) \  & \ H^3_c(A) \ &\  \mathrm{Im}(\Pi^3_{w_0}) \ & 0 & \ \cdots & \\
          2    &  \ H^2_c(G) \ &\  H^2_c(A) \  & \ \mathrm{Im}(\Pi^2_{w_0}) \ & 0 & \  \cdots &  \\
          1    & \ H^1_c(G) \ & \ H^1_c(A) \ & \  \mathrm{Im}(\Pi^1_{w_0}) \  & 0 & \   \cdots & \\
          0     & \ \mathbb{R}\   &\  \C_{c}(G/P)^G \ & \  C_{c,\alt}((G/P)^{(2)})^G  \ & \   C_{c,\alt}((G/P)^{(3)})^G \  &  \ \cdots & \\
    \quad\strut &   0  &  1  &  2  &  3  &  \cdots & p \strut \\};


\draw[->](m-3-2.east) -- (m-3-3.west)node[midway,above ]{$0$};
\draw[->](m-3-3.east) -- (m-3-4.west)node[midway,above ]{$\Pi^3_{w_0}$};
\draw[->](m-3-4.east) -- (m-3-5.west);

\draw[->](m-4-2.east) -- (m-4-3.west)node[midway,above ]{$0$};
\draw[->](m-4-3.east) -- (m-4-4.west)node[midway,above ]{$\Pi^2_{w_0}$};
\draw[->](m-4-4.east) -- (m-4-5.west);

\draw[->](m-5-2.east) -- (m-5-3.west)node[midway,above ]{$0$};
\draw[->](m-5-3.east) -- (m-5-4.west)node[midway,above ]{$\Pi^1_{w_0}$};
\draw[->](m-5-4.east) -- (m-5-5.west);

\draw[->](m-6-2.east) -- (m-6-3.west)node[midway,above ]{$\delta$};
\draw[->](m-6-3.east) -- (m-6-4.west)node[midway,above ]{$\delta$};
\draw[->](m-6-4.east) -- (m-6-5.west)node[midway,above ]{$\delta$};

\draw[thick] (m-1-1.east) -- (m-7-1.east) ;
\draw[thick] (m-7-1.north) -- (m-7-7.north) ;
\end{tikzpicture}
\caption{The first page $E_1$}\label{fig first page unbounded}
\end{figure}

The differential $d_1=d^\rightarrow$ on the first row boils down to the usual homogeneous differential $\delta$. For $q \geq 1$, the differential 
$$
d_1:H^q_c(G) \cong E^{0,q}_1 \longrightarrow E^{1,q}_1 \cong H^q_c(A)
$$
coincides with the restriction map multiplied by $(-1)^{p+1}$. By \cite[Corollary 3]{Wienhard} the restriction map vanishes identically, thus the same holds for the differentials from the first column to the second one. Finally, the same proof as \cite[Theorem 3]{BucSavAlt} shows that the differentials from the second column to the third one are conjugate to the operator $\Pi^q_{w_0}$. We depict the first page $E_1$ in Figure \ref{fig first page unbounded}.

The second page $E_2$ is now easily computed. The isomorphism $H^0_c(G \curvearrowright G/P) \cong \mathbb{R}$ implies that $E^{0,0}_2 \cong E^{1,0}_2 \cong 0$. For $p \geq 2$, on the row $q=0$ appears the alternating cohomology on $(G/P)^{(p)}$ shifted by one, that is 
$$
E^{p,0}_2 \cong H^{p-1}_{c,\mathrm{alt}}(G \curvearrowright G/P). 
$$
The continuous cohomology of $G$ on the first column $p=0$ is preserved, namely $E^{0,q}_2 \cong H^q_c(G)$. By the surjectivity of the differentials from the column $q=1$ to the column $q=2$, we obtain that 
$$
E^{1,q}_2 \cong H^q_c(A)^{w_0}, \ \ \ E^{2,q}_2 \cong 0. 
$$
We report the second page $E_2$ in Figure \ref{fig second page unbounded}. 

\begin{figure}[!h]
\centering
\begin{tikzpicture}
  \matrix (m) [matrix of math nodes,
             nodes in empty cells,
             nodes={minimum width=9ex,
                    minimum height=9ex,
                    outer sep=-3pt},
             column sep=2ex, row sep=-1ex,
             text centered,anchor=center]{
	q      &         &          &          & \\
          \cdots & \cdots & \cdots & \cdots   & \cdots \\
          3    & \ H^3_c(G) \  & \ H^3_c(A)^{w_0} \ &\  0 \ & 0 & \ \cdots & \\
          2    &  \ H^2_c(G) \ &\  H^2_c(A)^{w_0} \  & \  0 \ & 0 & \  \cdots &  \\
          1    & \ H^1_c(G) \ & \ H^1_c(A)^{w_0} \ & \  0 \  & 0 & \   \cdots & \\
          0     & \ 0\   &\  0 \ & \  H^1_{c,\alt}(G \curvearrowright G/P)  \ & \  H^2_{c,\alt}(G \curvearrowright G/P) \  &  \ \cdots & \\
    \quad\strut &   0  &  1  &  2  &  3  &  \cdots & p \strut \\};


\draw[->](m-3-2) -- (m-4-4);
\draw[->](m-3-3) -- (m-4-5);

\draw[->](m-4-2) -- (m-5-4);
\draw[->](m-4-3) -- (m-5-5);

\draw[->](m-5-2) -- (m-6-4);
\draw[->](m-5-3) -- (m-6-5);

\draw[thick] (m-1-1.east) -- (m-7-1.east) ;
\draw[thick] (m-7-1.north) -- (m-7-7.north) ;
\end{tikzpicture}
\caption{The second page $E_2$} \label{fig second page unbounded}
\end{figure}

Since we do not know if the spectral sequence degenerates or not, we will exploit now the map between bicomplexes
\begin{equation} \label{eq bicomplexes}
C^{p,q}_{\mathrm{alt}} \longrightarrow M^{p,q}_{\mathrm{alt}},
\end{equation}
where $M^{p,q}_{\mathrm{alt}}$ is the sub-bicomplex of boundary alternating cochains of the bicomplex given by Equation \eqref{meas:bicomplex}. We denote by
$$
E^{p,q}_{1,m}:=(H^q(M^{\ast,p}_{\mathrm{alt}},d^\uparrow),d_1=d^\rightarrow).
$$
the first page of the spectral sequence generated by the bicomplex on the right-hand side of Equation \eqref{eq bicomplexes}. The page $E_{1,m}$ is reported in \cite[Figure 3]{BucSavAlt}. The maps determined by Equation \eqref{eq bicomplexes} induce maps
\begin{equation}\label{eq first pages}
E_1^{p,q} \longrightarrow E_{1,m}^{p,q}
\end{equation}
for every $p,q \geq 0$. We claim that those maps are isomorphisms on each column for $q \geq 1$. When $p=0$ this is precisely the statement of Austin and Moore \cite[Theorem A]{AM}. Let us prove that the map $E^{1,q}_1 \rightarrow E^{1,q}_{1,m}$ induced up to conjugation by the inclusion 
\begin{equation}\label{eq inclusion continuous}
C_c(G^{q+1})^P \longrightarrow L^0(G^{q+1})^P,
\end{equation}
is an isomorphism. Notice that the inclusion of Equation \eqref{eq inclusion continuous} can be restricted to $A$ invariant cochains obtaining the following commutative diagram
$$
\xymatrix{
C_c(G^{q+1})^P \ar[rr] \ar[d] && L^0(G^{q+1})^P \ar[d] \\
C_c(G^{q+1})^A \ar[rr] && L^0(G^{q+1})^A.
}
$$
By \cite[Proposition 3.1]{Monod} the columns of the previous diagram are isomorphisms in cohomology. Additionally, the inclusion on the bottom row can be decomposed as
$$
C_c(G^{q+1})^A \rightarrow C_c(A^{q+1})^A \rightarrow L^0(A^{q+1})^A \rightarrow L^0(G^{q+1})^A,
$$
where the first function is the usual restriction map, the second one is the inclusion and the third one is induced by the $A$-equivariant projection $\pi_A:G \rightarrow A$, where $\pi_A$ is defined in terms of the Iwasawa decomposition $G=ANK$ as $\pi_A(ank)=a$, for $a\in A,\ n\in N$ and $k\in K$.  By \cite[Theorem A]{AM} and \cite[Theorem 2]{Moore} all of these maps induce isomorphisms in cohomology, thus the claim follows. The same argument also applies to the column $p=2$, so 
$$
E^{2,q}_1 \longrightarrow E^{2,q}_{1,m}
$$
are all isomorphisms, again for $q \geq 1$. 

We now pass to the second pages, where the page $E_{2,m}$ is depicted in \cite[Figure 4]{BucSavAlt}. The existence of the maps in  \eqref{eq first pages} implies that there exists a natural map 
\begin{equation}\label{eq second pages}
E^{p,q}_2 \longrightarrow E^{p,q}_{2,m},
\end{equation}
for every $p, q \geq 0$. Again for $q \geq 1$ these maps are all isomorphisms by \cite[Theorem 1]{AM}. As a consequence we obtain a commutative diagram
\begin{equation}\label{diagram A}
\xymatrix{
H^{q-1}_c(A)^{w_0} \ar[rr] \ar[rd] && H^q_{c,\alt}(G \curvearrowright G/P) \ar[ld]\\
& H^q_{m,\alt}(G \curvearrowright G/P), &
}
\end{equation}
where $H^q_{m,\alt}(G \curvearrowright G/P)$ is the cohomology of the cocomplex $(L^0_{\alt}((G/P)^{\ast+1})^G,\delta)$. By \cite[Theorem 3]{BucSavAlt} the map 
$$
H^{q-1}_c(A)^{w_0} \longrightarrow H^q_{m,\alt}(G \curvearrowright G/P)
$$
is injective, so the map 
$$
H^{q-1}_c(A)^{w_0} \longrightarrow H^q_{c,\alt}(G \curvearrowright G/P)
$$
is injective as well by the commutativity of Diagram \eqref{diagram A}. If we now pass to the quotients, the maps given by \eqref{eq second pages} determine another commutative diagram
\begin{equation}\label{eq diagram G}
\xymatrix{
H^q_c(G) \ar[rr] \ar[rd] && H^q_{c,\alt}(G \curvearrowright G/P)/H^{q-1}_c(A)^{w_0} \ar[ld] \\
& H^q_{m,\alt}(G \curvearrowright G/P)/H^{q-1}_c(A)^{w_0}. &
}
\end{equation}
Since \cite[Theorem 3]{BucSavAlt} guarantees that 
$$
H^q_c(G) \longrightarrow H^q_{m,\alt}(G \curvearrowright G/P)/H^{q-1}_c(A)^{w_0}
$$
is an isomorphism, we must have that 
$$
H^q_c(G) \longrightarrow H^q_{c,\alt}(G \curvearrowright G/P)/H^{q-1}_c(A)^{w_0}
$$
is injective and that 
$$
H^q_{c,\alt}(G \curvearrowright G/P)/H^{q-1}_c(A)^{w_0} \longrightarrow H^q_{m,\alt}(G \curvearrowright G/P)/H^{q-1}_c(A)^{w_0}
$$
is surjective. In particular 
$$
H^q_{c,\alt}(G \curvearrowright G/P)\longrightarrow H^q_{m,\alt}(G \curvearrowright G/P)
$$
also is surjective. 
\end{proof}

\section{Proof of Theorem \ref{thm_continuous_bounded}} \label{sec bounded case}
We follow the same strategy as for Theorem \ref{thm_continuous} in the previous section. We construct a bicomplex analogous to the one in  \eqref{bicomplex Kpq}, but this time using continuous bounded functions. More precisely, we define
$$
C_{b}^{p,q}:=C_{c,b}(G^{p+1} \times (G/P)^{(q)})^G, 
$$
where $(G/P)^{(q)}$ is the space of generic $q$-tuples. The vertical differential 
$$
d^\uparrow:C^{p,q}_b \longrightarrow C^{p+1,q}_b
$$
is simply the homogeneous differential on the $G$-variable, whereas the horizontal one
$$
d^\rightarrow:C^{p,q}_b \longrightarrow C^{p,q+1}_b
$$
is the homogeneous differential on $(G/P)^{(q)}$ weighted with the sign $(-1)^{p+1}$. The restriction of the alternation map $\mathrm{Alt}_q$ given by Equation \eqref{eq alternation continuous} to the subset of bounded cochains leads to the following splitting
$$
C^{p,q}_b \cong C^{p,q}_{b,\mathrm{alt}} \oplus C^{p,q}_{b,\mathrm{n-alt}},
$$
similar to the one we had in the unbounded case. A first important difference with respect to the bicomplex $C^{p,q}_b$ is that we have no exponential law for bounded functions, so it is not true that $C^{p,q}_b$ is isomorphic to $C_{c,b}(G^{p+1},C_{c,b}((G/P)^{(q)}))^G$ (and similarly for the alternating cochains). 

The bicomplex $(C_b^{p,q},d^\rightarrow, d^\uparrow)$ determines two different spectral sequences. In contrast to the unbounded setting, the fact that we are now dealing with bounded functions, allows us to prove vanishing of the  first spectral sequence
\begin{equation}\label{eq first spectral bounded}
^I_bE^{p,q}_1:=(H^q(C^{p,\ast}_{b,\mathrm{alt}},d^{\rightarrow}),d_1=d^\uparrow). 
\end{equation}

\begin{prop}\label{first bounded zero}
Let $^I_bE^{p,q}_1$ be the first page of the spectral sequence defined by \eqref{eq first spectral bounded}. Then the spectral sequence degenerates immediately, that is $^I_bE^{p,q}_1=0$ for every $p,q \geq 0$. 
\end{prop}

\begin{proof}
We will first show that the first page of the spectral sequence defined by 
$$
(H^q(C^{p,\ast}_{b},d^{\rightarrow}),d_1=d^\uparrow)
$$
degenerates. The result will follow from the fact that $^I_bE^{p,q}_1$ is a direct summand of that vector space. 

The first page of the spectral sequence is obtained by considering the horizontal differential $d^\rightarrow$ in $C^{p,q}_b$, namely it is obtained by the cohomology of the cocomplex
\begin{equation}\label{eq cocomplex horizontal}
\rightarrow C_{c,b}(G^{p+1} \times (G/P)^{(q-1)})^G \rightarrow C_{c,b}(G^{p+1} \times (G/P)^{(q)})^G \rightarrow .
\end{equation}
The evaluation on the neutral element allows us to get rid of $G$-invariance by deleting one variable and to study the \emph{inhomogeneous cocomplex}, namely
\begin{equation}\label{eq inhomogeneous bounded}
\rightarrow C_{c,b}(G^{p}\times (G/P)^{(q-1)}) \rightarrow C_{c,b}(G^{p} \times (G/P)^{(q)}) \rightarrow .
\end{equation}
It is worth noticing that the differential on the boundary variable is still the homogeneous one, suitably weighted with a sign. The cocomplex of Equation \eqref{eq inhomogeneous bounded} admits a contracting homotopy given by integration along the first boundary variable (with respect to some probability measure) and hence it is exact. This implies that the first page of the spectral sequence
$$
(H^q(C^{p,\ast}_{b},d^{\rightarrow}),d_1=d^\uparrow)
$$
degenerates immediately. Since $^I_bE^{p,q}_1$ is a direct summand of $H^q(C^{p,\ast}_b,d^\rightarrow)$, this concludes the proof. 
\end{proof} 

We now turn to the second spectral sequence, namely the one with first page
$$
^{II}_bE_1^{p,q}:=(H^q(C_{b,\mathrm{alt}}^{\ast,p},d^\uparrow),d_1=d^\rightarrow). 
$$
Also in this context we have that $^{II}_bE^{p,q}_b$ is a direct summand of the spectral sequence whose first page is given by 
$$
(H^q(C_{b}^{\ast,p},d^\uparrow),d_1=d^\rightarrow). 
$$
Our aim is to identify all the terms in the page $^{II}_bE_1$. We are going to prove that all the entries which are not either on the first column or on the first row must vanish. In order to do that we will obtain an analogous statement for $H^q(C_{b}^{\ast,p},d^\uparrow)$.

\begin{prop}\label{cor second third columns}
For every $q \geq 1$ we have that 
$$
^{II}_b E^{1,q}_1 \cong ^{II}_bE^{2,q}_1 \cong 0
$$
\end{prop}

\begin{proof}
By what we have said so far it is sufficient to show that 
$$
H^q(C_{b}^{\ast,1},d^\uparrow) \cong H^q(C_{b}^{\ast,2},d^\uparrow) \cong 0,
$$
for every $q \geq 1$. 

We start with the case of only one boundary point. Following the same strategy we exploited in the unbounded setting, we know that there exists an isomorphism of cocomplexes
$$
\Psi:C_{c,b}(G^{q+1})^P \rightarrow C_{c,b}(G^{q+1} \times (G/P))^G, \ \ \Psi(f)(g_0,\ldots,g_q,gP):=f(g^{-1}g_0,\ldots,g^{-1}g_q).
$$
which induces an isomorphism
$$
H^q(C_{b}^{\ast,1},d^\uparrow) \cong H^q_{c,b}(P).
$$
The amenability of $P$ implies that the cohomology $H^q(C_{b}^{\ast,1},d^\uparrow)$ vanishes identically when $q \geq 1$ by \cite[Corollary 7.5.11]{Monodbook}.

By substituting $P$ with the subgroup $MA$ we similarly have that
$$
H^q(C_b^{\ast,2},d^\uparrow) \cong H^q_{c,b}(MA),
$$
which vanishes again by the amenability of $MA$. This concludes the proof. 
\end{proof}

We are left to show that, for $q \geq 1$, all the remaining columns but the first one are trivial.

\begin{thm}\label{thm vanishing gp bounded}
For every $q \geq 1$ and $p \geq 3$ we have that 
$$
H^q(C_{b}^{\ast,p},d^\uparrow) \cong 0. 
$$
\end{thm} 

\begin{proof}
The strategy is the same as Theorem \ref{thm vanishing gp unbounded}. First, it is worth noticing that the same argument that we exploited to show Lemma \ref{K invariant representative} can be adapted in the bounded context to prove that the cocomplex 
$$
(_bC_K^{p,q},d^\uparrow):=((C_{c,b}((G/K)^{p+1} \times (G/P)^{(q)})^G,d^\uparrow),
$$
computes the cohomology group $H^q(C_{b}^{\ast,p},d^\uparrow)$. As a consequence it is sufficient to build a contracting homotopy for the cocomplex $_bC_K^{p,q}$. The same homotopy defined in the proof of Theorem \ref{thm vanishing gp unbounded} works fine, since it preserves boundedness and continuity. 
\end{proof}

\begin{cor}\label{cor bounded 3 points}
For every $q \geq 1$ and $p \geq 3$ we have that
$$
^{II}_bE^{p,q}_1 \cong 0.
$$
\end{cor}

\begin{proof}
The group $^{II}_b E^{p,q}_1$ is a direct summand of $H^q(C_{b}^{\ast,p},d^\uparrow)$, which vanishes by Theorem \ref{thm vanishing gp bounded}.
\end{proof}

\begin{proof}[Proof of Theorem \ref{thm_continuous_bounded}]
We consider the first page of the spectral sequence 
$$
^{II}_bE^{p,q}_1:=(H^q(C^{\ast,p}_{b,\mathrm{alt}},d^\uparrow),d_1:=d^\rightarrow)
$$
This page is particularly easy: in fact by Proposition \ref{cor second third columns} and Corollary \ref{cor bounded 3 points} we have that $^{II}_bE^{p,q}_1 \cong 0$ for every $p \geq 1$ and $q \geq 1$. We are left only with the first column and the first row. On the first column appears the bounded cohomology of $G$, namely
$^{II}_bE^{0,q}_1 \cong H_{c,b}^q(G)$. In a similar way, on the bottom row we get back continuous bounded alternating functions on $(G/P)^{(p)}$, that is
$^{II}_bE^{p,0}_1 \cong C_{c,b,\alt}((G/P)^{(p)})^G$. Additionally the differential $d_1=d^\rightarrow$ boils down to the usual homogeneous differential $\delta$ on $(G/P)^{(p)}$. We report the first page $^{II}_bE^{p,q}_1$ in Figure \ref{first page bounded}.

\begin{figure}[!h]
\centering
\begin{tikzpicture}
  \matrix (m) [matrix of math nodes,
             nodes in empty cells,
             nodes={minimum width=9ex,
                    minimum height=9ex,
                    outer sep=-3pt},
             column sep=2ex, row sep=-1ex,
             text centered,anchor=center]{
         q      &         &          &          & \\
          \cdots & \cdots & \cdots & \cdots   & \cdots \\
          3    & \ H^3_{c,b}(G) \  & \ 0 \ &\  0 \ & 0 & \ \cdots & \\
          2    &  \ H^2_{c,b}(G) \ &\  0 \  & \ 0 \ & 0 & \  \cdots &  \\
          1    & \ H^1_{c,b}(G) \ & \ 0 \ & \  0 \  & 0 & \   \cdots & \\
          0     & \ \mathbb{R}\   &\  \C_{c,b}(G/P)^G \ & \  C_{c,b,\alt}((G/P)^{(2)})^G  \ & \   C_{c,b,\alt}((G/P)^{(3)})^G \  &  \ \cdots & \\
    \quad\strut &   0  &  1  &  2  &  3  &  \cdots & p \strut \\};


\draw[->](m-3-2.east) -- (m-3-3.west)node[midway,above ]{$0$};
\draw[->](m-3-3.east) -- (m-3-4.west)node[midway,above ]{$0$};
\draw[->](m-3-4.east) -- (m-3-5.west);

\draw[->](m-4-2.east) -- (m-4-3.west)node[midway,above ]{$0$};
\draw[->](m-4-3.east) -- (m-4-4.west)node[midway,above ]{$0$};
\draw[->](m-4-4.east) -- (m-4-5.west);

\draw[->](m-5-2.east) -- (m-5-3.west)node[midway,above ]{$0$};
\draw[->](m-5-3.east) -- (m-5-4.west)node[midway,above ]{$0$};
\draw[->](m-5-4.east) -- (m-5-5.west);

\draw[->](m-6-2.east) -- (m-6-3.west)node[midway,above ]{$\delta$};
\draw[->](m-6-3.east) -- (m-6-4.west)node[midway,above ]{$\delta$};
\draw[->](m-6-4.east) -- (m-6-5.west)node[midway,above ]{$\delta$};

\draw[thick] (m-1-1.east) -- (m-7-1.east) ;
\draw[thick] (m-7-1.north) -- (m-7-7.north) ;
\end{tikzpicture}
\caption{The first page $^{II}_bE_1$}\label{first page bounded}
\end{figure}

The second page $^{II}_bE^{p,q}_2$ is simple as well: the only thing to modify is the bottom row, where we now have the cohomology on the boundary shifted by $1$, namely 
$$
^{II}_bE^{p,0}_2 \cong H_{c,b,\alt}^{p-1}(G \curvearrowright G/P). 
$$
We report the second page $^{II}_bE^{p,q}_2$ in Figure \ref{second page bounded}.

\begin{figure}[!h]
\centering
\begin{tikzpicture}
  \matrix (m) [matrix of math nodes,
             nodes in empty cells,
             nodes={minimum width=9ex,
                    minimum height=9ex,
                    outer sep=-3pt},
             column sep=2ex, row sep=-1ex,
             text centered,anchor=center]{
         q      &         &          &          & \\
          \cdots & \cdots & \cdots & \cdots   & \cdots \\
          3    & \ H^3_{c,b}(G) \  & \ 0 \ &\  0 \ & 0 & \ \cdots & \\
          2    &  \ H^2_{c,b}(G) \ &\  0 \  & \  0 \ & 0 & \  \cdots &  \\
          1    & \ H^1_{c,b}(G) \ & \ 0 \ & \  0 \  & 0 & \   \cdots & \\
          0     & \ 0\   &\  0 \ & \  H^1_{c,b,\alt}(G \curvearrowright G/P)  \ & \  H^2_{c,b,\alt}(G \curvearrowright G/P) \  &  \ \cdots & \\
    \quad\strut &   0  &  1  &  2  &  3  &  \cdots & p \strut \\};


\draw[->](m-3-2) -- (m-4-4);
\draw[->](m-3-3) -- (m-4-5);

\draw[->](m-4-2) -- (m-5-4);
\draw[->](m-4-3) -- (m-5-5);

\draw[->](m-5-2) -- (m-6-4);
\draw[->](m-5-3) -- (m-6-5);

\draw[thick] (m-1-1.east) -- (m-7-1.east) ;
\draw[thick] (m-7-1.north) -- (m-7-7.north) ;
\end{tikzpicture}
\caption{The second page $^{II}_bE_2$} \label{second page bounded}
\end{figure}
 
Since the spectral sequence $^I_bE_1$ degenerates immediately by Proposition \ref{first bounded zero}, by \cite[Appendix A]{Guichardet} the spectral sequence $^{II}_bE_1$ must converge to the same limit. The only way to obtain this convergence is that
$
d_{q+1}:H^q_{c,b}(G) \longrightarrow H^q_{c,b,\alt}(G \curvearrowright G/P)
$
is an isomorphism. 

Following the proof of Theorem \ref{thm_continuous}, we can consider the bounded analogue of the bicomplex $M^{p,q}$, namely
$$
M^{p,q}_b=L^\infty(G^{p+1} \times (G/P)^q)^G,
$$
with the same vertical and horizontal differentials, and similarly for the alternating bicomplex $M^{p,q}_{b,\mathrm{\alt}}$. We clearly have an inclusion
$$
C^{p,q}_{b,\alt} \longrightarrow M^{p,q}_{b,\mathrm{alt}},
$$
which is a map of bicomplexes. We denote by
$$
^{II}_bE^{p,q}_{1,m}:=(H^q(M^{\ast,p}_{b,\mathrm{alt}},d^\uparrow),d_1=d^\rightarrow)
$$
the first page of the second spectral sequence generated by the measurable bicomplex. The page $^{II}_bE_{1,m}$ is the same as the one reported in Figure \ref{first page bounded} up to substituting continuous cochains with measurable ones. The natural maps 
$$
^{II}_bE^{p,q}_{1} \rightarrow ^{II}_bE^{p,q}_{1,m}
$$
are isomorphisms for every $p \geq 0$ and $q \geq 1$ by \cite[Proposition 7.5.1]{Monodbook}. By passing to the second pages, for every $q \geq 0$, we obtain the following diagram
$$
\xymatrix{
H^q_{c,b}(G) \ar[rr]^{d_{q+1,m}} \ar[rd]^{d_{q+1}} && H^q_{m,b,\alt}(G \curvearrowright G/P) \\
& H^q_{c,b,\alt}(G \curvearrowright G/P), \ar[ru] &
}
$$
where $H^q_{m,b,\mathrm{alt}}(G \curvearrowright G/P)$ is the cohomology of the cocomplex $(L^\infty((G/P)^{\ast+1})^G,\delta)$ and $d_{q+1,m}$ is the differential of order $p+1$ in the measurable bicomplex. Since we proved that $d_{q+1}$ is an isomorphism and $d_{q+1,m}$ is an isomorphism by \cite{Monod}, the statement follows.  
\end{proof}

\section{Proof of Proposition 3}

A direct consequence of the proofs of Theorem \ref{thm_continuous} and \ref{thm_continuous_bounded} is that we have the following commutative diagram
\begin{equation}\label{diagram comparison}
\xymatrix{
H^q_{c,b}(G) \ar[d]_{\mathrm{comp}^q_G} \ar@{^{(}->>}[rr]^{\hspace{-40pt} d_\infty}_\cong && H^q_{c,b,\alt}(G \curvearrowright G/P) \ar[d]^{p\circ \mathrm{comp}^q_{G \curvearrowright G/P}}\\
H^q_c(G) \ar@{^{(}->}[rr]^{\hspace{-40pt} d_\infty} && H^q_{c,\alt}(G \curvearrowright G/P)/H^{k-1}_c(A)^{w_0},
}
\end{equation}
where $p:H^q_{c,\alt}(G \curvearrowright G/P)\rightarrow H^q_{c,\alt}(G \curvearrowright G/P)/H^{q-1}_c(A)^{w_0}$ is the natural projection and we tacitly identify $H^{q-1}_c(A)^{w_0}$ with the image $d_\infty(H^{q-1}_c(A)^{w_0})\subset H^q_{c,\alt}(G \curvearrowright G/P)$. The functions $d_\infty$ appearing in the top and bottom arrows are determined by the differentials of the associated spectral sequences and  are both injective. 

Since the bottom $d_\infty$ is injective, the injectivity of $\mathrm{comp}^q_G$ is equivalent to the injectivity of the composition $d_\infty \circ  \mathrm{comp}^q_G= p\circ \mathrm{comp}^q_{G \curvearrowright G/P}\circ d_\infty$. Now the top $d_\infty$ is furthermore known to be an isomorphism so the latter map is injective if and only if $p\circ \mathrm{comp}^q_{G \curvearrowright G/P}$ is. The latter claim is equivalent to $\mathrm{comp}^q_{G \curvearrowright G/P}$ being injective, and for the kernel of $p$ to intersect the image of $\mathrm{comp}^q_{G \curvearrowright G/P}$ trivially, which is precisely the statement of the Proposition. 

\section{Injectivity of the comparison map}\label{sec injectivity}

We are finally ready to prove the main injectivity results of the paper.

\begin{proof}[Proof of Theorem \ref{thm boundary deg 4}] We first deal with the case $n=2$. We identify the boundary of the hyperbolic plane with the real projective line $\mathbb{P}^1(\mathbb{R})$. Recall that in the real rank one case generic tuples in $\mathbb{P}^1(\mathbb{R})^{(k)}$ simply are distinct tuples.

Consider a $\mathrm{PSL}(2,\mathbb{R})$-invariant continuous bounded alternating cocycle
$$
b:\mathbb{P}^1(\mathbb{R})^{(5)} \longrightarrow \mathbb{R}. 
$$
Suppose that there exists a $\mathrm{PSL}(2,\mathbb{R})$-invariant continuous alternating function 
$$
f:\mathbb{P}^1(\mathbb{R})^{(4)} \longrightarrow \mathbb{R}
$$
such that $\delta f=b$. We will show that $f$ is bounded, which implies the desired injectivity. We define
$$\begin{array}{rrcl}
F:&\mathbb{P}^1(\mathbb{R}) \setminus \{\infty,0,1\}& \longrightarrow& \mathbb{R},\\
&x&\longmapsto &f(\infty,0,1,x).
\end{array}$$
We will actually show that $F$ is bounded, which readily implies that $f$ is bounded as well by the transitivity of $\mathrm{PSL}(2,\mathbb{R})$ on positively oriented triples and the fact that $f$ is alternating. Recall that for distinct $x_i\in P^1(\mathbb{R})$, by the $\mathrm{PSL}(2,\mathbb{R})$-invariance of $f$, we have that 
$$f(x_0,x_1,x_2,x_3)=f(\infty,0,1,[x_0,x_1,x_2,x_3])=F([x_0,x_1,x_2,x_3]),$$
 where 
$$[x_0,x_1,x_2,x_3]=\frac{x_0-x_2}{x_0-x_3}\cdot \frac{x_1-x_3}{x_1-x_2}$$
denotes the usual cross ratio. 

By the continuity of $F$, we only need to prove its boundedness in some neighborhood of $\infty$, $0$ and $1$. In fact exploiting again the fact that $f$ is alternating it will suffice to prove boundedness of $F$ on some interval $(1-\delta,1)$, for any $\delta>0$. Let us see how this implies boundedness of $F$ on the three neighborhoods: For $x \in (1,1+\delta']$ it is sufficient to notice that 
\begin{equation}\label{x to 1/x}
F(x)=f(\infty,0,1,x)=-f(\infty,0,x,1)=-f\left(\infty,0,1,\frac{1}{x}\right)=-F\left(\frac{1}{x}\right)
\end{equation}
where we exploited the fact that $f$ is alternating. The latter term is bounded since $1/x<1$ lies in a neighborhood of $1$.  The function $F$ is thus uniformly bounded around $1$ (where the latter value is excluded). For the boundedness around $0$ we have
$$F(x)=f(\infty,0,1,x)=-f(\infty,1,0,x)=-F(1-x),$$
so that the boundedness around $0$ is equivalent to the boundedness around $1$. Finally for the boundedness around $\infty$ we use the boundedness around $0$ and the relation (\ref{x to 1/x})


Let $0<\delta<1$. We will show that $F$ is bounded on $(1-\delta,1)$. 
We apply the relation $\delta f =b$ to a $5$-tuple $(\infty,0,1,x,y)$ of distinct points
to obtain
\begin{equation}\label{F cocycle}
F(x)-F(y)+F\left(\frac{y}{x}\right)-F\left( \frac{1-y}{1-x} \right) + F\left( \frac{x(1-y)}{y(1-x)} \right) \sim 0,
\end{equation}
where the symbol $\sim$ means that the left-hand side is uniformly at bounded distance from the right-hand side. Indeed the left-hand side is equal to $b(\infty,0,1,x,y)$, which is bounded by our assumption on $b$. 

We now plug into  Equation \eqref{F cocycle} the value $y=x^2$ (which is allowed since for $x\in \mathbb{R}\setminus \{0,1\}$ the $5$-tuple $(\infty,0,1,x,x^2)$ consists of distinct points) to obtain
\begin{equation}\label{eq x minus x}
2F(x)-F(x^2)-F(1+x)+F\left(\frac{1+x}{x}\right) \sim 0.
\end{equation}
If $x$ lies in a neighborhood of $1$, then $1+x$ and $(1+x)/x$ both lie in a neighborhood of $2$ where $F$ is continuous and hence bounded. In particular, for $x\in [1-\delta,1)$, the expression \eqref{eq x minus x} rewrites as 
\begin{equation}\label{eq F twice}
F(x^2) \sim 2F(x),
\end{equation}
More precisely, there must exist  $C>0$ such that for every $x \in [1-\delta,1)$ it holds that 
\begin{equation}\label{eq C}
\frac{F(x^2)}{2}-C \leq F(x) \leq \frac{F(x^2)}{2}+C. 
\end{equation}
Since 
$$
[1-\delta,1)=\bigcup_{k \geq 0} [(1-\delta)^{\frac{1}{2^k}},(1-\delta)^{\frac{1}{2^{k+1}}}),
$$
there exists some $k\geq 0$ such that $x \in [(1-\delta)^{\frac{1}{2^k}},(1-\delta)^{\frac{1}{2^{k+1}}})$. By \eqref{eq C} we have  
$$
\frac{F(x^2)}{2}-C \leq F(x) \leq \frac{F(x^2)}{2}+C,
$$
with $x^2 \in [(1-\delta)^{\frac{1}{2^{k-1}}},(1-\delta)^{\frac{1}{2^k}})$. If $k> 0$ we can apply \eqref{eq C} again to obtain
$$
\frac{F(x^4)}{4}-C\left(1+\frac{1}{2}\right) \leq F(x) \leq \frac{F(x^4)}{4}+C\left(1+\frac{1}{2}\right),
$$
with $x^4 \in [(1-\delta)^{\frac{1}{2^{k-2}}},(1-\delta)^{\frac{1}{2^{k-1}}})$. Finally we iterate this procedure $k+1$ times, until we obtain 
$$
\frac{F(x^{2^{k+1}})}{2^{k+1}}-C\left(1+\frac{1}{2}+\ldots+\frac{1}{2^k}\right) \leq F(x) \leq \frac{F(x^{2^{k+1}})}{2^{k+1}}+C\left(1+\frac{1}{2}+\ldots+\frac{1}{2^k}\right),
$$
with $x^{2^k} \in [(1-\delta)^2,1-\delta)$. The continuity of $F$ on the closure of $[(1-\delta)^2,1-\delta)$ ensures that $F(x^{2^k})$ is bounded. Hence the convergence of the geometric series guarantees that $F(x)$ is bounded for $x \in [1-\delta,1)$ and concludes the proof for $n=2$.

When $n=3$ the argument is similar:  We identify the Furstenberg boundary with the complex projective line $\mathbb{P}^1(\mathbb{C})$. As before, we suppose that a $\mathrm{PSL}(2,\mathbb{C})$-invariant continuous bounded alternating cocycle 
$$
c:\mathbb{P}^1(\mathbb{C})^{(5)} \longrightarrow \mathbb{R}
$$
can be written as $\delta f=c$, where $f$ is a $G$-invariant continuous alternating function
$$
f:\mathbb{P}^1(\mathbb{C})^{(4)} \longrightarrow \mathbb{R}.
$$
We define 
$$\begin{array}{rrcl}
F: &\mathbb{P}^1(\mathbb{C}) \setminus \{\infty,0,1 \} &\longrightarrow& \mathbb{R}\\
&x&\longmapsto &f(\infty,0,1,x),
\end{array}$$
it is sufficient to prove the boundedness of $F$. As in the real case, it will be enough to show the boundedness of $F$ on a neighborhood of $1$. The same equations as in the $n=2$ case imply the boundedness of $F$ around $0$ and $\infty$ as well. Furthermore by (\ref{x to 1/x}) is enough to prove boundedness in the intersection of a neighborhood of $1$ and the closed unit disc centered at $1$. 

Fix some small $\delta>0$ and let $U$ be  given as
$$U=\{z\in \mathbb{C}\setminus \{1\} \mid 1-\delta <|z|\leq 1, \ -\delta < \mathrm{arg}(z)<\delta\}.$$
Precisely as in the $n=2$ case, for $x\in U$ we have
$$
2F(x)-F(x^2) \sim 0.
$$
Thus applying the same inductive estimates  it is enough to observe that for every $z\in U$ there exists $k$ such that $z^{2^k}$ belongs to the closure of 
$$\{ z\in \mathbb{C}\setminus \{1\} \mid (1-\delta)^2<|z|\leq 1, -2\delta < \mathrm{arg}(z) < 2\delta\} \setminus U,$$
on which $F$ is bounded by continuity. This concludes the proof for $n=3$.

The case $n \geq 4$ follows immediately from the case $n=3$: Let 
$$
b:(\partial \mathbb{H}^n_{\mathbb{R}})^{(5)} \longrightarrow \mathbb{R}
$$
be a $G$-invariant continuous bounded alternating cocycle and suppose that it can be written as $\delta f=b$, where $f$ is a $G$-invariant continuous alternating function
$$
f: (\partial \mathbb{H}^n_{\mathbb{R}})^{(4)} \longrightarrow \mathbb{R}. 
$$
Fix a copy of $\partial \mathbb{H}^3_{\mathbb{R}} \subset \partial \mathbb{H}^n_{\mathbb{R}}$. By the case $n=3$, we know that the restriction $f|_{(\partial \mathbb{H}^3_{\mathbb{R}})^{(4)}}$ is bounded. This implies that $f$ is bounded as well: an arbitrary $4$-tuple $(x_0,x_1,x_2,x_3)$ of boundary points lies in an $\mathrm{Isom}^+(\mathbb{H}^n)$-translate of $\partial \mathbb{H}^3$, so there exists $g \in \mathrm{Isom}^\circ(\mathbb{H}^n_{\mathbb{R}})$ such that $gx_0,gx_1,gx_2,gx_3 \in \partial \mathbb{H}^3_{\mathbb{R}}$, and the claim follows by the $G$-invariance of $f$. This concludes the proof.
\end{proof}


\bibliographystyle{alpha}

\end{document}